\documentclass[12pt]{amsart}
\usepackage{amscd,amssymb,graphics}

\usepackage{amsfonts}
\usepackage{amsmath}
\usepackage{amsxtra}
\usepackage{latexsym}
\usepackage[mathcal]{eucal}
\usepackage[all]{xy}
\usepackage[T1]{fontenc}
\usepackage{lmodern}

\usepackage{graphics,colortbl}

\input xy
\xyoption{all}
\usepackage{epsfig}

\usepackage[pdftex,bookmarks,colorlinks,breaklinks]{hyperref}

\newtheorem{theorem}{Theorem}[section]

\newtheorem{proposition}[theorem]{Proposition}
\newtheorem{corollary}[theorem]{Corollary}
\newtheorem{lemma}[theorem]{Lemma}
\newtheorem{claim}
{Claim}

\theoremstyle{definition}
\newtheorem{definition}[theorem]{Definition}

\newtheorem{question}[theorem]{Question}
\theoremstyle{remark}

\newtheorem{example}[theorem]{Example}
\newtheorem{remark}[theorem]{Remark}

\usepackage{vmargin}
\setmarginsrb{20mm}{10mm}{20mm}{16mm}%
{20mm}{7mm}{20mm}{8mm}
\catcode`\@=12

\def\Ne{{\mathcal N}}

\def\U{{\mathcal U}}

\def\UT{\mathbf{UT}}
\def\NT{\mathbf{NT}}

\def\f{\operatorname{fi}}

\def\add{\operatorname{add}}
\def\cof{\operatorname{cof}}
\def\lh{\operatorname{lh}}
\def\sym{\operatorname{sym}}

\def\Sp{\operatorname{Spec}}

\input xy  \xyoption{all}

\begin{document}
	
	
	
	
	\title[Cofinal types of topological groups]{Cofinal types of topological groups}

\author[X. Gong]{Xuan Gong}
 \address[X. Gong]
{Institute of Mathematics, Nanjing Normal University, Nanjing 210024, China}
\email{gongxuan2021@163.com}

\author[D. Peng]{Dekui Peng\textsuperscript{*}}
 \address[D. Peng]
{Institute of Mathematics, Nanjing Normal University, Nanjing 210024, China}
\email{pengdk10@lzu.edu.cn}
\thanks{*Corresponding author.}

\begin{abstract}
We investigate the local structure of topological groups via Tukey order and introduce the \emph{fineness index} $\f(P)$ for directed posets. This invariant generalizes the bounding number $\mathfrak b$ and gives the character bound $\chi(G)\in\{1,\omega\}\cup[\f(P),\cof(P)]$ for every topological group $G$ with a $P$-base. We also prove that $\mathcal N_e(F(X,\mathcal U))=_T\mathcal U^\omega$ for every non-discrete $\add(\mathcal U)$-precompact uniform space $(X,\mathcal U)$. In particular, this applies to every infinite compact space. Finally, under Shelah's Strong Hypothesis (\textbf{SSH}), we show that both $\mathcal N_e(F(X,\mathcal U))$ and $\mathcal N_e(A(X,\mathcal U))$ are Tukey equivalent to $[\tau]^{<\omega}$ whenever $(X,\mathcal U)$ is an infinite precompact uniform space of weight $\tau$, $\cof(\tau)>\omega$, and $\tau\ge\mathfrak d$.
\end{abstract}

\subjclass[2020]{54H11, 03E04, 54A25} 

\keywords{Cofinal type, free topological group, Tukey order, cofinality, character of topological groups}
	\maketitle	
\section{Introduction}

Throughout the paper, all topological groups are assumed to be Hausdorff. Uniform spaces are separated, as in condition (iv) below, and compact spaces are assumed to be Hausdorff.

Neighborhood bases at the identity of topological groups connect general topology, topological algebra, and combinatorial set theory. Given a topological group $G$, the neighborhood filter at the identity, denoted by $\mathcal{N}_e(G)$ and ordered by reverse inclusion, forms a directed partially ordered set (poset). Its Tukey type provides a natural way to compare local structures of topological groups. For two directed posets $P$ and $Q$, we say that $P$ is Tukey below $Q$, denoted by $P \le_T Q$, if there exists a map from $Q$ to $P$ carrying cofinal sets to cofinal sets.

The local structure of non-metrizable topological groups has received considerable attention, particularly for groups admitting a $P$-base (i.e., when there exists a cofinal, order-preserving map from a directed poset $P$ to $\mathcal{N}_e(G)$); see, for example, \cite{DF, Feng, Huang, KM}. An important instance is the class of topological groups with an $\omega^\omega$-base \cite{Ban, BL, GKL, LPT, LRZ}. Following work on uniform spaces by Cascales and Orihuela \cite{CO}, Gabriyelyan, K\k{a}kol, and Leiderman \cite{GKL} showed that the character $\chi(G)$ of a non-metrizable group with an $\omega^\omega$-base belongs to the interval $[\mathfrak{b}, \mathfrak{d}]$. Leiderman, Pestov, and Tomita \cite{LPT} subsequently showed that the lower bound can be attained by constructing groups whose neighborhood bases are Tukey equivalent to prescribed regular cardinals Tukey below $\omega^\omega$.

The first objective of this paper is to define a cardinal invariant for general directed posets that gives lower bounds for character and pseudocharacter. We introduce the \emph{fineness index}, denoted by $\f(P)$, which is the first cardinal at which domination by countable subsets fails. We prove that every topological group $G$ with a $P$-base satisfies $\chi(G) \in \{1, \omega\} \cup [\f(P), \cof(P)]$, extending the corresponding bound for $\omega^\omega$.

The second objective concerns the exact cofinal types of free topological groups. Nickolas and Tkachenko \cite{NT,NT2} developed uniform covering trees to compute the character of free topological groups over $\omega$-narrow spaces, including compact spaces, and Chiș, Ferrer, Hern\'{a}ndez, and Tsaban \cite{CFHT} obtained related extensions and pcf evaluations. In particular, the character equality $\chi(F(X, \mathcal{U})) = \cof(\mathcal{U}^\omega)$ is known. We ask whether this cardinal equality can be strengthened to the Tukey equivalence $\mathcal{N}_e(F(X, \mathcal{U})) =_T \mathcal{U}^\omega$.

We give a positive answer for a broad class of uniform spaces. Using a levelwise refinement of uniform covering trees, called \emph{neat trees}, we prove that $\mathcal{N}_e(F(X, \mathcal{U})) =_T \mathcal{U}^\omega$ for every non-discrete $\operatorname{add}(\mathcal{U})$-precompact uniform space $(X, \mathcal{U})$. In particular, the equivalence holds for every infinite compact space. The new point here is the Tukey equivalence itself, rather than the previously known equality of characters.

Finally, we invoke Shelah's Strong Hypothesis ({\bf SSH}). Arising from pcf (possible cofinalities) theory, {\bf SSH} asserts that the pseudo-power $pp(\lambda) = \lambda^+$ for every singular cardinal $\lambda$. Utilizing {\bf SSH}, we prove that for an infinite precompact uniform space of weight $\tau$ with $\cof(\tau) > \omega$ and $\tau\ge\mathfrak{d}$, both $\mathcal{N}_e(F(X, \mathcal{U}))$ and $\mathcal{N}_e(A(X, \mathcal{U}))$ are Tukey equivalent to $[\tau]^{<\omega}$.
	\label{sec0}

	\subsection{Tukey order of (pre-)ordered sets}
	
	Let $P$ be a set. A \emph{quasi-order} (or \emph{pre-order}) $\leq$ on $P$ is a reflexive and transitive binary relation. If $\leq$ is also antisymmetric, then it is called a \emph{partial order}. In this case, the pair $(P,\leq)$ is called a \emph{partially ordered set} (briefly, a \emph{poset}). 
	
	A pre-ordered set $(P,\leq)$ is called \emph{directed} if for every $x,y\in P$ there exists $z\in P$ such that $x\leq z$ and $y\leq z$. 
	
	For $A,B\subseteq P$, we say that $A$ is \emph{dominated} by $B$ if for every $a\in A$ there exists $b\in B$ such that $a\leq b$.
	
	Let $(P,\leq)$ be a directed pre-ordered set. A subset $A\subseteq P$ is called
	\begin{itemize}
		\item \emph{cofinal} in $P$ if $P$ is dominated by $A$;
		\item \emph{bounded} in $P$ if $A$ is dominated by a single point of $P$.
	\end{itemize}
	
	We associate two cardinals with $P$:
	\begin{itemize}
		\item the \emph{cofinality} 
		\[
		\cof(P)=\min\{\,|A|: A \text{ is cofinal in } P\,\};
		\]
		\item the \emph{additivity} 
		\[
		\add(P)=\min\{\,|A|: A \text{ is unbounded in } P\,\}.
		\]
	\end{itemize}
	
	Note that a pre-ordered set may have no nonempty unbounded subset; this happens precisely when $P$ has a (not necessarily unique) largest element. In this case, we put $\add(P)=1$, which is equivalent to $\cof(P)=1$.
	
	It is easy to see that if $\add(P)\neq 1$, then $\add(P)$ is a regular cardinal and $\add(P)\leq \cof(P)$. The well-known cardinal characteristics $\mathfrak{b}$ and $\mathfrak{d}$ are precisely the additivity and the cofinality of $(\omega^\omega,\leq^*)$, respectively, where $f\leq^* g$ means that $f(n)\leq g(n)$ for all but finitely many $n\in\omega$.
	
	We now introduce a central notion in the theory of directed posets, due to Tukey \cite{Tuk}. Let $(P,\leq_P)$ and $(Q,\leq_Q)$ be directed pre-ordered sets. We say that $Q$ is \emph{Tukey above} $P$, or equivalently that $P$ is \emph{Tukey below} $Q$, and write $P\leq_T Q$, if there exists a \emph{cofinal map} $\varphi:Q\to P$, that is, a map which sends cofinal subsets of $Q$ to cofinal subsets of $P$. 
	
	By a characterization of Schmidt \cite{Sch, Tod}, this is equivalent to the existence of an \emph{unbounded map} $\psi:P\to Q$, that is, a map sending unbounded subsets of $P$ to unbounded subsets of $Q$.
	
	If $P$ is Dedekind complete (i.e., every bounded subset of $P$ has a least upper bound), then $P\leq_T Q$ if and only if there exists an order-preserving map $f:Q\to P$ such that $f(Q)$ is cofinal in $P$ \cite{GartM}.
	
	Two posets $P$ and $Q$ are said to be \emph{Tukey equivalent}, written $P=_T Q$, if $P\leq_T Q$ and $Q\leq_T P$. A \emph{cofinal type} is a Tukey equivalent class.
	We list several well-known examples, which will also be used throughout the paper, to better understand the Tukey order.
	
	\begin{example}\label{ex}\cite{Mam}
		In the following, $P$ and $Q$ are directed pre-ordered sets.
		
		\begin{itemize}
			\item[(i)] (folklore) The only countable cofinal types are $1$ and $\omega$.\footnote{The classification of cofinal types of cardinality at most $\omega_1$ is sensitive to additional set-theoretic assumptions; see \cite{Isb, Tod}.}
			
			\item[(ii)]
			$
			\cof(\cof(P))=_T\cof(P)\le_T P
			\quad\text{and}\quad
			\add(P)\le_T P.
			$
			
			\item[(iii)] 
			$P =_T Q$ if and only if both can be embedded as cofinal subsets of a common directed preordered set. In particular, every cofinal subset of $P$ is Tukey equivalent to $P$.
			
			\item[(iv)] 
			The product $P\times Q$ is the least upper bound of $\{P,Q\}$ in the Tukey order. Hence the collection of cofinal types forms an upper semilattice.
			
			\item[(v)] 
			If $\kappa$ is an uncountable cardinal and $\cof(P)\le \kappa$, then
			\[
			P\le_T ([\kappa]^{<\omega},\subseteq).
			\]
			
			\item[(vi)] If $\kappa<\add(P)$, then $P^\kappa=_T P$.
		\end{itemize}
		
	\end{example}
	
	\subsection{Uniform spaces}
	Let $X$ be a set and let $A$ and $B$ be subsets of $X\times X$, i.e., relations on the set $X$. Let 
	\[
	-A=\{(x,y):(y,x)\in A\}
	\]
	and 
	\[
	A+B=\{(x,z): \text{there exists a } y\in X \text{ such that }(x,y)\in A \text{ and }(y,z)\in B \}.
	\]
	For a relation $A\subseteq X\times X$ and a natural number $n$ the relation $nA\subseteq X \times X$ is defined inductively by the formulas
	\[
	1A=A \text{ and }nA=(n-1)A+A.
	\]
	
	The diagonal of the Cartesian product $X\times X$ is $\Delta=\{(x,x): x\in X\}$. Every set $V\subseteq X\times X$ that contains $\Delta$ and satisfies the condition $V=-V$ is called an entourage of the diagonal; the family of all entourages of the diagonal $\Delta $ will be denoted by $\mathcal{D}_X$. If for a pair $x,y\in X$ and a $V\in \mathcal{D}_X$ we have $(x,y)\in V$, we write $|x-y|<V$; otherwise we write $|x-y|\ge V$. 
	
	Let $x_0 \in X$ and $V\in \mathcal{D}_X$. The set $B(x_0,V)=\{x\in X: |x_0-x|< V\}$ is called the $V$-ball about $x_0$. 
	
	A uniformity on a set $X$ is a subfamily $\mathcal{U}$ of $\mathcal{D}_X$ which satisfies the following conditions:
	\begin{itemize}
		\item[(i)] If $V\in \mathcal{U}$ and $V\subseteq U\in \mathcal{D}_X$, then $U\in \mathcal{U}$.
		\item[(ii)] If $V_1,V_2\in \mathcal{U}$, then $V_1 \cap V_2 \in \mathcal{U}$.
		\item[(iii)] For every $V\in \mathcal{U}$ there exists a $U\in \mathcal{U}$ such that $2U\subseteq V$.
		\item[(iv)] $\cap \mathcal{U}=\Delta$.
	\end{itemize}
	
	A family $\mathcal{B}\subseteq \mathcal{U}$ is called a base for the uniformity $\mathcal{U}$ if for every $V\in \mathcal{U}$ there exists a $W\in \mathcal{B}$ such that $W\subseteq V$.
	\[
	w(\mathcal{U})=\text{min} \{|\mathcal{B}|:\mathcal{B} \text{ is a base for }\mathcal{U}\}
	\]
	is called the weight of the uniformity $\mathcal{U}$.
	
	A uniform space is a pair $(X,\mathcal{U})$ consisting of a set $X$ and a uniformity $\mathcal{U}$ on $X$. The weight of a uniform space $(X,\mathcal{U})$ is defined as the weight of the uniformity $\mathcal{U}$.
	
	Any cover of the uniform space $(X,\mathcal{U})$ which has a refinement of the form $\{B(x,V):x\in X\}$, where $V\in \mathcal{U}$, is called a uniform cover.
	
	\subsection{Free Topological Groups over Uniform Spaces}
	Let $(X, \mathcal{U})$ be a uniform space. The free topological group (respectively, free Abelian topological group) over $(X, \mathcal{U})$, denoted by $F(X, \mathcal{U})$ (respectively, $A(X, \U)$), is the free group (respectively, free Abelian group) generated by $X$, endowed with the finest group topology such that every uniformly continuous map from $(X, \mathcal{U})$ into a topological group (respectively, an Abelian topological group), equipped with its natural two-sided uniformity, uniquely extends to a continuous group homomorphism.

This framework naturally generalizes the classical Markov free topological group over a Tychonoff space $X$. Indeed, if $\mathcal{U}_X$ denotes the fine uniformity on $X$, then $F(X)$ is topologically isomorphic to $F(X, \mathcal{U}_X)$. For any compatible uniformity $\mathcal{U}$ on $X$, the relation $\mathcal{U} \subseteq \mathcal{U}_X$ implies that the identity map induces a continuous isomorphism from $F(X)$ onto $F(X, \mathcal{U})$; hence, the topology of $F(X, \mathcal{U})$ is coarser than or equal to that of $F(X)$. The parallel statements hold for free Abelian topological groups.

The uniform-space setting provides additional structure for studying the identity neighborhood filter. In particular, the entourages and uniform covers of $\mathcal U$ allow neighborhood bases to be organized by uniform covering trees \cite{NT,NT2}.

	\section{Fineness Index of Posets and \texorpdfstring{$P$}{P}-Base}
	
	\subsection{Fineness index: definition and basic computation}
	For a topological space $X$ and $x\in X$, we denote by $\Ne_x(X)$ the neighbourhood filter at $x$, ordered by reverse inclusion.
	
	Let $P$ be a directed poset. Following Dow and Feng \cite{DF}, we say that a topological group $G$ has a $P$-base if there exists an order-preserving map from $P$ to $\Ne_e(G)$ with cofinal image. Clearly, this is stronger than the condition $\Ne_e(G)\le_T P$. Since $\Ne_e(G)$ is always Dedekind complete, i.e., every bounded subset has a least upper bound, the two notions coincide; see \cite{GartM}.
	
	Among non-metrizable cases, the most extensively studied instance is $P=\omega^\omega$; see \cite{Ban, Feng, GKL,LPT}. Such groups are said to have a $\mathfrak G$-base. One defines uniform spaces with an $\omega^\omega$-base similarly. Such spaces were first studied in \cite{CO}, where, among other things, every compact space with an $\omega^\omega$-base was shown to be metrizable. For further results on topological spaces, groups, and uniform spaces with an $\omega^\omega$-base, we refer the reader to \cite{Ban}.
	
	Topological groups with an $\omega^\omega$-base were systematically studied in \cite{GKL} by Gabriyelyan,~K\k{a}kol and Leiderman. Among other results, it was shown that if a topological group admits an $\omega^\omega$-base, then its character, namely the cofinality of $\Ne_e(G)$, belongs to
	\[
	\{1,\omega\}\cup[\mathfrak b,\mathfrak d].
	\]

	This result may seem somewhat unexpected, since $\mathfrak b$ is not a cardinal invariant arising directly from $\omega^\omega$, but rather from $(\omega^\omega,\le^*)$, while the additivity of $\omega^\omega$ is countable. It is therefore natural to look for another intrinsically defined cardinal invariant of posets that can serve as a lower bound for non-metrizable topological groups with a $P$-base. The following notion provides such a tool.
	
	\begin{definition}
		A directed pre-ordered set $(P,\leq)$ is called \emph{$\kappa$-fine} if every subset 
		$A\subseteq P$ with $|A|\le \kappa$ is dominated by a countable subset of $P$.
		
		When $\cof(P)>\omega$, let $\f(P)$ denote the least cardinal $\lambda$ such that $P$ is not $\lambda$-fine. This cardinal is called the \emph{fineness index} of $P$.
		
		For completeness, we set $\f(P)=1$ if $\cof(P)=1$, and $\f(P)=\omega$ if $\cof(P)=\omega$.
	\end{definition}

	The definition is well posed. Indeed, every directed pre-ordered set is $\omega$-fine, since each countable subset dominates itself. On the other hand, if $\cof(P)>\omega$ and $A\subseteq P$ is cofinal with $|A|=\cof(P)$, then $A$ cannot be dominated by a countable subset of $P$. Consequently,
	$
	\omega_1\leq\f(P)\leq\cof(P)
	$
	whenever $\cof(P)>\omega$.
	
	\begin{remark}
		The cardinal invariant $\f(P)$ is closely related to $\add(P)$: informally, one obtains $\f(P)$ from the definition of $\add(P)$ by replacing ``$1$'' with ``$\omega$'' (when $\cof(P)>\omega$). Indeed,  $\add(P)$ is the least cardinality of a subset of $P$ that cannot be dominated by a subset of cardinality $1$.
		
		The notion of $\kappa$-fineness was first introduced for topological groups by He, Tkachenko, Zhang and the second-listed author \cite{HPTZ}: a topological group $G$ is called $\kappa$-fine if $\Ne_e(G)$ is $\kappa$-fine. The invariant $\f(G)$ was then introduced naturally in subsequent work \cite{ZHX}. The weakest such property, namely $\omega_1$-fineness, already deserves special attention; see, for example, \cite{LRS}, where $\omega_1$-filters are called \emph{$A$-filters}.
		
		In these papers, $\kappa$-fineness is shown to be closely related to an important topic in the theory of topological groups, namely \emph{$\mathbb R$-factorizable topological groups} \cite[Chap. 8]{AT}. 
	\end{remark}
	It is straightforward to see that if $P$ and $Q$ are Dedekind complete directed posets and $P\le_T Q$, then the $\kappa$-fineness of $Q$ implies the $\kappa$-fineness of $P$. The next lemma shows that the assumption of Dedekind completeness can in fact be omitted.
	
	\begin{lemma}\label{Tukey}
		Let $\kappa$ be uncountable, and let $P$ and $Q$ be directed posets such that $P\le_T Q$. If $Q$ is $\kappa$-fine, then so is $P$.
	\end{lemma}
	
	\begin{proof}
		Fix a map $\varphi:P\to Q$ sending unbounded subsets of $P$ to unbounded subsets of $Q$.
		Let $A\subseteq P$ with $|A|\le \kappa$. Since $Q$ is $\kappa$-fine, the set $\varphi(A)$ is dominated by a countable set $\{q_n:n\in\omega\}$ in $Q$.
		
		For each $n\in\omega$, define
		$
		A_n=\{a\in A:\varphi(a)\le q_n\}.
		$
		Then $\varphi(A_n)$ is bounded in $Q$, and by the choice of $\varphi$, the set $A_n$ must be bounded in $P$. Thus, for each $n$ one can choose $b_n\in P$ such that $a\le b_n$ for all $a\in A_n$.
		
		Since $A=\bigcup_{n\in\omega}A_n$, the countable set $\{b_n:n\in\omega\}$ dominates $A$. Hence $P$ is $\kappa$-fine.
	\end{proof}
	
	\begin{corollary}\label{finein} If $P\geq_T Q$ and $\cof(Q)>\omega$, then $\f(P)\leq \f(Q)$.
	\end{corollary}
	
	It is immediate that if $P$ has countable cofinality, then $P$ is $\kappa$-fine for every uncountable cardinal $\kappa$. Thus we are mainly interested in pre-ordered sets of uncountable cofinality. The following elementary lemma shows that the essential case is when the additivity is countable.
	
	\begin{lemma}\label{cont}
		Let $P$ be a directed pre-ordered set of uncountable cofinality. Then:
		\begin{itemize}
			\item[(i)] If $\add(P)>\omega$, then $\f(P)=\add(P)$;
			\item[(ii)] $\f(P)$ is regular.
		\end{itemize}
	\end{lemma}
	
	\begin{proof}
		(i) Assume $\add(P)>\omega$. Then every countable subset of $P$ is bounded. Consequently, any subset of $P$ dominated by a countable set is in fact bounded, i.e., dominated by a single point of $P$. Therefore $P$ is $\kappa$-fine precisely when $\kappa<\add(P)$, so $\f(P)=\add(P)$.
		
		(ii) Put $\lambda=\f(P)$, and suppose toward a contradiction that $\lambda$ is singular. Since $P$ is not $\lambda$-fine, there is a set $X\subseteq P$ with $|X|\le\lambda$ which is not dominated by any countable subset of $P$. In fact $|X|=\lambda$: otherwise $P$ would fail to be $|X|$-fine, contradicting the minimality of $\lambda$. Let $\kappa=\cof(\lambda)<\lambda$, and write
		$
		X=\bigcup_{\alpha<\kappa}X_\alpha
		$
		where $|X_\alpha|<\lambda$ for every $\alpha<\kappa$.
		
		For each $\alpha<\kappa$, since $|X_\alpha|<\lambda$, there exists a countable set $A_\alpha\subseteq P$ dominating $X_\alpha$. Then
		$
		A=\bigcup_{\alpha<\kappa}A_\alpha
		$
		has cardinality at most $\kappa$, and dominates $X$. Since $\kappa<\lambda$ and $P$ is $\kappa$-fine, the set $A$ is dominated by a countable subset of $P$, and hence so is $X$, a contradiction.
	\end{proof}
	
	\begin{corollary}
		$\f(\omega^\omega)=\mathfrak b$.
	\end{corollary}
	
	\begin{proof}
		The identity map from $(\omega^\omega,\le)$ to $(\omega^\omega,\le^*)$ is cofinal, and hence
		$
		(\omega^\omega,\le^*)\le_T(\omega^\omega,\le).
		$
		Therefore, by Corollary~\ref{finein} and Lemma~\ref{cont}(i),
		\[
		\f(\omega^\omega)\le \f(\omega^\omega,\le^*)=\mathfrak b.
		\]
		
		For the reverse inequality, let $\kappa<\mathfrak b$, and let 
		$A\subseteq \omega^\omega$ satisfy $|A|\le \kappa$.
		Then $A$ is bounded in $(\omega^\omega,\le^*)$, so there exists 
		$g\in\omega^\omega$ such that $f\le^* g$ for every $f\in A$.
		
		Now note that $f\le^* g$ holds iff there exist $m,n\in\omega$ such that
		\[
		f\le g\vee c_{m,n},
		\]
		where $c_{m,n}\in\omega^\omega$ is defined by
		$c_{m,n}(k)=m$ for $k\le n$, and $c_{m,n}(k)=0$ otherwise.
		
		It follows that every $f\in A$ is dominated (with respect to the pointwise order) by some member of the countable set
		\[
		\{\,g\vee c_{m,n}:m,n\in\omega\,\}.
		\]
		Thus $A$ is dominated by a countable subset of $\omega^\omega$, so $\omega^\omega$ is $\kappa$-fine.
		
		Therefore $\f(\omega^\omega)\ge \mathfrak b$, and hence
		$\f(\omega^\omega)=\mathfrak b$.
	\end{proof}

	\begin{theorem}\label{Th:car}
		Let $P$ be a directed poset. Then every topological group with a $P$-base satisfies
		\[
		\chi(G)\in \{1,\omega\}\cup [\f(P),\cof(P)].
		\]
	\end{theorem}
	
	\begin{proof}
		If $G$ is discrete, then $\chi(G)=1$. If $G$ is non-discrete and first-countable, then $\chi(G)=\omega$. Thus it remains to consider the case where $G$ is not first-countable. Then $\chi(G)$ is uncountable and
		\[
		\chi(G)=\cof(\Ne_e(G)).
		\]
		
		Since $G$ has a $P$-base, we have $\Ne_e(G)\le_T P$. Hence, by Corollary~\ref{finein},
		\[
		\f(P)\le \f(\Ne_e(G)).
		\]
		Moreover,
		\[
		\f(\Ne_e(G))\le \cof(\Ne_e(G))\le \cof(P).
		\]
		Combining these inequalities yields
		\[
		\chi(G)=\cof(\Ne_e(G))\in [\f(P),\cof(P)].
		\]
		
		Together with the trivial cases $\chi(G)=1$ and $\chi(G)=\omega$, the conclusion follows.
	\end{proof}
	
	We shall see that an analogous result also holds for the pseudocharacter.
	Recall that the pseudocharacter of a topological group $G$ is the least cardinal $\tau$ such that
	$
	\{e\}
	$
	is the intersection of $\tau$ many open sets; we denote it by $\psi(G)$.
	\begin{proposition}
		Let $P$ be a directed poset. Then every topological group with a $P$-base satisfies
		\[
		\psi(G)\in \{1,\omega\}\cup [\f(P),\cof(P)].
		\]
	\end{proposition}
	
	\begin{proof}
		Since $\psi(G)\le \chi(G)$, it suffices to prove:
		
		\smallskip
		
		\emph{If $\psi(G)>\omega$, then $\psi(G)\ge \f(P)$.}
		\smallskip
		
		Let $\kappa=\psi(G)$, and suppose
		$
		\{e\}=\bigcap_{\alpha<\kappa} U_\alpha,
		$
		where each $U_\alpha$ is an open neighbourhood of $e$.
		Suppose $\kappa<\f(P)$. Then $P$ is $\kappa$-fine, and hence so is $\Ne_e(G)$ by Lemma~\ref{Tukey}. Therefore there exists a countable family of neighbourhoods $\{V_n:n\in\omega\}$ such that each $U_\alpha$ contains some $V_n$. It follows that
		$
		\{e\}=\bigcap_{n\in\omega}V_n,
		$
		so $\psi(G)\le \omega$, a contradiction.
	\end{proof}

	If one defines, for any directed poset $P$,
	\[
	\Sp_\chi(P)=\{\chi(G): \text{$G$ is a non-metrizable topological group with a $P$-base}\},
	\]
	called the \emph{character spectrum of topological groups with a $P$-base}, then Theorem~\ref{Th:car} yields
	\[
	\Sp_\chi(P)\subseteq [\f(P),\cof(P)].
	\]
	When $P=\omega^\omega$, this recovers precisely the result of Gabriyelyan, K\k akol and Leiderman.
	
	They asked whether there exists a topological group with an $\omega^\omega$-base whose character is exactly $\mathfrak b$. This was soon answered positively by Leiderman, Pestov, and Tomita, who proved the following stronger result.
	
	\begin{proposition}\cite[Proposition 2.6]{LPT}
		If $\kappa$ is a regular cardinal Tukey below $\omega^\omega$, then there exists a topological group $G$ such that
		\[
		\Ne_e(G)=_T \kappa.
		\]
	\end{proposition}
	
	Since
	\[
	\mathfrak b\le_T (\omega^\omega,\le^*)\le_T \omega^\omega
	\]
	(see Example~\ref{ex}(ii)), this immediately yields a positive answer.
	
	It is therefore natural to ask whether 
\begin{question}\label{Q1} Let $P$ be a directed poset with $\cof(P)>\omega$. Do we have $\f(P)\in \Sp_\chi(P)$? Or equivalently, is $\f(P)$ the least cardinal that can occur as the character of a non-metrizable topological group with a $P$-base?
\end{question}

 We now give some discussion of this question.
	
	We begin with a simple observation on constructing topological groups with prescribed cofinal type (together with some information about pseudocharacter). 
	
	\begin{proposition}\label{poset2group}
		Let $P$ be a directed poset. Then there exists a Boolean group $G$ with a local base $\mathcal{B}$ at $e$ consisting of subgroups and
		\[
		\Ne_e(G)=_T P,
		\qquad
		\psi(G)=\add(P).
		\]
	\end{proposition}
	
	\begin{proof}
		If $P=_T1$, take $G$ discrete. So assume $P\neq_T1$.
		
		Let
		\[
		G=([P]^{<\omega},\triangle),
		\]
		where $\triangle$ denotes symmetric difference. For each $p\in P$, let
		\[
		\uparrow p=\{q\in P:q\ge p\},
		\qquad
		G_p=[\uparrow p]^{<\omega}.
		\]
		
		It is straightforward that $\{G_p:p\in P\}$ forms a neighbourhood base at the identity for a group topology on $G$, and
		$
		p\mapsto G_p
		$
		is an order embedding with cofinal image. Hence
		$
		\Ne_e(G)=_T P.
		$
		Moreover,
		$
		\bigcap_{p\in P}G_p=\{\emptyset\},
		$
		so the topology is Hausdorff.
		
		It remains to check that $\psi(G)=\add(P)$. For every $A\subseteq P$, we have
		\[
		\bigcap_{p\in A}G_p
		=
		\left[\bigcap_{p\in A}\uparrow p\right]^{<\omega}.
		\]
		Consequently, $\bigcap_{p\in A}G_p=\{\emptyset\}$ if and only if $A$ is unbounded in $P$. Choose an unbounded set $A\subseteq P$ with $|A|=\add(P)$. Then
		\[
		\bigcap_{p\in A}G_p=\{\emptyset\},
		\]
		and hence $\psi(G)\le\add(P)$.
		
		Conversely, suppose that $\{\emptyset\}=\bigcap_{\alpha<\lambda}U_\alpha$, where each $U_\alpha$ is a neighborhood of $\emptyset$. Choose $p_\alpha\in P$ such that $G_{p_\alpha}\subseteq U_\alpha$. Then
		\[
		\bigcap_{\alpha<\lambda}G_{p_\alpha}
		\subseteq
		\bigcap_{\alpha<\lambda}U_\alpha
		=\{\emptyset\}.
		\]
		Thus $\{p_\alpha:\alpha<\lambda\}$ is unbounded in $P$, and therefore $\lambda\ge\add(P)$. Hence $\psi(G)=\add(P)$.
	\end{proof}
	
	By the proposition, an uncountable cardinal $\kappa$ belongs to $\Sp_\chi(P)$ if and only if
	$
	\kappa=\cof(Q)
	$
	for some directed poset $Q\le_T P$. In particular, $\kappa\in\Sp_\chi(P)$ implies $\kappa\le_T P$. The converse cannot be obtained in this way for a singular $\kappa$, since the ordinal $\kappa$, regarded as a directed set, has cofinality $\cof(\kappa)<\kappa$. If $\kappa$ is regular and $\kappa\le_T P$, however, Proposition~\ref{poset2group}, applied to $Q=\kappa$, produces a non-metrizable group of character $\kappa$ with a $P$-base. Therefore, for every uncountable regular cardinal $\kappa$,
	\[
	\kappa\in \Sp_\chi(P)
	\quad\Longleftrightarrow\quad
	\kappa\le_T P.
	\]
	
	Since $\f(P)$ is regular, Question \ref{Q1} is equivalent to the following.
	
	\begin{question}\label{Q:f(P)}
		Do we have $\f(P)\le_T P$?
	\end{question}

	\subsection{The Countable Power}
	In this subsection, we study $\kappa$-fineness of the countable power of a poset under the coordinatewise order. 
	
	%
	%
	%
	%
	%
	
	
We first present a result closely related to Question \ref{Q:f(P)}.
	\begin{proposition}\label{finepower}
		For every directed poset $P$, one has
		\[
		\f(P)\le_T P^\omega .
		\]
	\end{proposition}
	
	\begin{proof}
		If $\cof(P)=1$, the assertion is trivial. If $\cof(P)=\omega$, then $P=_T\omega$ and
		\[
		\f(P)=\omega\le_T\omega^\omega=_T P^\omega.
		\]
		Thus assume that $\cof(P)>\omega$.
		
		If $\add(P)>\omega$, then by Lemma~\ref{cont},
		\[
		\f(P)=\add(P)\le_T P =_T P^\omega,
		\]
		and we are done.
		
		Suppose now that $\add(P)=\omega$. Let
		$
		\kappa=\f(P),
		$
		and fix a family $\{p_\alpha:\alpha<\kappa\}\subseteq P$ which is not dominated by any countable subset of $P$.
		
		For each $f\in P^\omega$, let $\varphi(f)$ be the least ordinal $\alpha<\kappa$ such that $p_\alpha$ is not dominated by
		\[
		G_f=\{f(n):n\in\omega\}.
		\]
		
		We claim that $\varphi:P^\omega\to\kappa$ is a cofinal map.
		
		First, $\varphi$ is order-preserving. Indeed, if $f\le g$ in $P^\omega$, then $G_g$ dominates $G_f$, so any initial segment dominated by $G_f$ is also dominated by $G_g$. Hence
		\[
		\varphi(f)\le \varphi(g).
		\]
		
		Next, $\varphi$ has cofinal image. Fix $\alpha<\kappa$. Since $P$ is $|\alpha|$-fine (because $|\alpha|<\f(P)$), there exists a countable set $\{a_n:n\in\omega\}$ dominating
		$
		\{p_\beta:\beta<\alpha\}.
		$
		Define $f\in P^\omega$ by $f(n)=a_n$. Then $G_f$ dominates every $p_\beta$ for $\beta<\alpha$, so the first element not dominated by $G_f$ must occur at some stage at least $\alpha$, i.e.
		$
		\varphi(f)\ge \alpha.
		$
		Thus the image of $\varphi$ is cofinal in $\kappa$.
		
		Therefore $\kappa\le_T P^\omega$, i.e.
		\[
		\f(P)\le_T P^\omega.              ~~~\qedhere
		\]
	\end{proof}

	Now we have a partial positive answer to Question~\ref{Q:f(P)}: 
	
	If
	$
	P^\omega=_T P
	$
	—for example, if $P=\omega^\omega$—then
	$
	\f(P)\le_T P.
	$

	\begin{proposition}\label{omegapower}
		For every directed poset $P$, one has
		\begin{itemize}
			\item[(i)] $\f(P^\omega)=\f(P)$, if $\add(P)>\omega$;
			\item[(ii)] $\f(P^\omega)=\min\{\mathfrak b, \f(P)\}$, if $\omega=\add(P)<\cof(P)$;
			\item[(iii)] $\f(P^\omega)=\mathfrak{b}$ if $\cof(P)=\omega$.
		\end{itemize}
	\end{proposition}
	
	\begin{proof}
		For (i), Example~\ref{ex}(vi) gives $P^\omega=_T P$ whenever $\add(P)>\omega$, and hence $\f(P^\omega)=\f(P)$.
		
		For (ii) and (iii), since $\omega^\omega\le_T P^\omega$ and $P\leq_T P^\omega$, Corollary~\ref{finein} yields
		\[
		\begin{aligned}
			\f(P^\omega) &\le \f(\omega^\omega)=\mathfrak b,\\
			\f(P^\omega) &\le \f(P), \qquad\text{if }\cof(P)>\omega.
		\end{aligned}
		\]

		To prove the reverse inequality, let $\kappa<\min\{\mathfrak b, \f(P)\}$ in (ii) and respectively $\kappa<\mathfrak{b}$ in (iii), and let
		$A\subseteq P^\omega$ have cardinality at most $\kappa$.
		For each $n\in\omega$, let
		\[
		A_n=\{x(n):x\in A\}\subseteq P.
		\]
		Then $|A_n|\le \kappa$, so $A_n$ is dominated by a countable subset $B_n\subseteq P$, because $P$ is $\kappa$-fine in each case of (ii) and (iii). Enumerate $B_n$ as $\{b_{n,k}:k\in\omega\}$, allowing repetitions when $B_n$ is finite. By directedness, choose an increasing sequence $(c_{n,k})_{k\in\omega}$ in $P$ such that
		\[
		b_{n,0},\ldots,b_{n,k}\le c_{n,k}
		\]
		for every $k\in\omega$. Put $C_n=\{c_{n,k}:k\in\omega\}$.
		
		The set
		$
		C=\prod_{n\in\omega}C_n
		$
		dominates $A$. Moreover, the map
		\[
		\theta:\omega^\omega\longrightarrow C,
		\qquad
		\theta(f)(n)=c_{n,f(n)},
		\]
		is order-preserving and has cofinal image.
		
		For each $x\in A$, choose $c_x\in C$ with $x\le c_x$, and then choose $f_x\in\omega^\omega$ such that $c_x=\theta(f_x)$. Let $D=\{f_x:x\in A\}$. Then $|D|\le\kappa$.
		
		Since $\omega^\omega$ is $\kappa$-fine, the set $D$ is dominated by a countable subset $H\subseteq\omega^\omega$. Consequently, the countable set $\theta(H)$ dominates $\{c_x:x\in A\}$, and hence also dominates $A$ in $P^\omega$.
		
		Thus $P^\omega$ is $\kappa$-fine for every $\kappa<\min\{\mathfrak b, \f(P)\}$ in (ii) and for every $\kappa<\mathfrak{b}$, so
		the assertions of (ii) and (iii) follow.
	\end{proof}

	\section{Cofinal Type of Free Topological (Abelian) Groups}
	If $(X,\U)$ is a uniform space, then $\U$, viewed as a subset of the power set of $X\times X$, can be ordered by reverse inclusion, i.e., $U\leq V$ if and only if $U\supseteq V$. With this order, $\U$ becomes a directed poset.
	
	
	The set of all uniformly continuous pseudometrics on $(X,\U)$ bounded by $1$ is denoted by $\mathcal{P}(X,\U)$, and is ordered by $d_1\leq d_2$ if $d_1(x,y)\leq d_2(x,y)$ for all $x,y\in X$.
	
	The following description of the local Tukey structure of free Abelian topological groups was made explicit by Gartside in \cite[Theorem 3.1]{Gart}. It is worth noting, however, that this result had essentially already been proved much earlier by Nickolas and Tkachenko in \cite{NT}. Their original goal was to compute the character of free Abelian topological groups, but their argument in fact establishes the stronger Tukey equivalence below. This was only explicitly observed later by Gartside in 2021.
	
	\begin{lemma}\cite{Gart,NT}\label{3eq}
		Let $(X,\U)$ be a uniform space. Then
		\[
		\Ne_e(A(X,\U))=_T (\U^\omega,\le)=_T \mathcal P(X,\U).
		\]
	\end{lemma}
	
	Combining this with Proposition~\ref{omegapower}, we obtain the following.
	
	\begin{corollary}
		Let $(X,\U)$ be a uniform space. Then
		\begin{itemize}
			\item $\f(A(X,\U))=1$, if $\cof(\U)=1$;
			\item $\f(A(X,\U))=\f(\U)$, if $\add(\U)>\omega$;
			\item $\f(A(X,\U))=\min\{\mathfrak b,\f(\U)\}$, if $\add(\U)=\omega<\cof(\U)$;
			\item $\f(A(X,\U))=\mathfrak b$, if $\cof(\U)=\omega$.
		\end{itemize}
	\end{corollary}

	By \cite[Theorems 2.9 and 3.15]{NT}, Nickolas and Tkachenko computed the character of the free topological group $F(X,\U)$ over an $\omega$-narrow uniform space, proving that
	\[
	\chi(F(X,\U))=\chi(A(X,\U))=\cof(\U^\omega).
	\]
	Related extensions and pcf evaluations were subsequently obtained by Chiș, Ferrer, Hernández and Tsaban \cite{CFHT}. These are character computations; the new issue considered below is whether the equality can be lifted to an equivalence of Tukey types.
	
	However, a careful inspection shows that their argument establishes an equality of characters, but not explicitly the stronger statement of Tukey equivalence. Before proving the above theorem, they first analyzed the case where $\U$ has a countable base. In fact, \cite[Lemma 3.9]{NT} proves:
	
	\medskip
	\emph{If $\cof(\U) \leq \omega$ for an $\omega$-narrow uniform space $(X,\U)$, then
		\[
		\chi(F(X,\U))\le \mathfrak d.
		\]}
	\medskip
	Moreover, one easily checks that equality holds whenever $(X,\U)$ is non-discrete, equivalently, whenever $\cof(\U)\neq 1$. As in the Abelian case, their proof actually yields the stronger statement:
	
	\begin{theorem}[Nickolas--Tkachenko]
		If $(X,\U)$ is an $\omega$-narrow uniform space with $\U=_T\omega$, then
		\[
		\Ne_e(F(X,\U))=_T\omega^\omega.
		\]
	\end{theorem}

	\begin{proof}
		The proof of \cite[Lemma 3.9]{NT}, especially parts (II)--(V), gives more than the stated character bound. Let $D$ denote the poset appearing in equation (3.6) there; it is order-isomorphic to $\omega^\omega$. The constructions in those parts assign to each $d\in D$ a sequence $P(d)$ of uniform covering trees such that
		$
		d\leq d'\quad\Longrightarrow\quad O_{P(d)}\supseteq O_{P(d')},
		$
		and every identity neighborhood contains $O_{P(d)}$ for some $d\in D$. Thus $d\mapsto O_{P(d)}$ is an order-preserving map from $D$ to $\Ne_e(F(X,\U))$, ordered by reverse inclusion, with cofinal image. Hence
		$
		\Ne_e(F(X,\U))\le_T\omega^\omega.
		$
		Conversely, $A(X,\U)$ is a quotient of $F(X,\U)$, so Lemma~\ref{3eq} gives
		$
		\Ne_e(F(X,\U))\ge_T\Ne_e(A(X,\U))=_T\U^\omega=_T\omega^\omega.
		$
	\end{proof}
	
	This naturally leads to the question whether their cofinality formula can be strengthened to a Tukey equivalence.
	
	\begin{question}\label{MainQ}
		Let $(X,\U)$ be an $\omega$-narrow uniform space. Must
		\[
		\Ne_e(F(X,\U))=_T \U^\omega?
		\]
	\end{question}
	
	Our next goal is to give a positive answer for a class of uniform spaces that includes all compact spaces, by modifying the proof of Nickolas and Tkachenko. We first recall the required uniform-covering-tree machinery.
	
	We begin with some basic notions from combinatorial set theory. Let $A$ be a nonempty set and let $n\in\omega\setminus\{0\}$. We denote by $A^n$ the set of sequences
	\[
	s=(s(0),s(1),\dots,s(n-1))
	\]
	of length $n$ with values in $A$, viewed as functions from $n$ to $A$.
	
	For $n=0$, we let $A^0=\{\emptyset\}$, the singleton consisting of the \emph{empty sequence}. The length of a finite sequence $s$ is denoted by $\lh(s)$, so in particular $\lh(\emptyset)=0$.
	
	If $s\in A^n$ and $m\leq n$, we write
	\[
	s\restriction m=(s(0), s(1), \dots, s(m-1))
	\]
	(hence $s\restriction 0=\emptyset$).
	We say that $t$ is an \emph{initial segment} of $s$, and that $s$ is an \emph{extension} of $t$, if $t=s\restriction m$ for some $m\leq \lh(s)$.
	
	If $s\in A^n$ and $a\in A$, we denote by
	\[
	s^\frown a=(s(0), s(1), \dots, s(n-1), a)\in A^{n+1}
	\]
	the concatenation of $s$ with $a$. In particular, $\emptyset^\frown a=(a)\in A^1$.
	
	As usual, we denote by $A^{<\omega}$ the set of all finite sequences from $A$, i.e.,
	\[
	A^{<\omega}=\bigcup_{n\in\omega} A^n.
	\]
	
	A \emph{tree} on $A$ is a subset $T\subseteq A^{<\omega}$ closed under taking initial segments. 
	
	\begin{definition}\label{def}
		Let $(X,\U)$ be a uniform space, and denote by $\mathcal{T}$ the family of all nonempty open subsets of $X$. A \emph{uniform covering tree} is a tree $T$ on $\mathcal{T}$ such that for every $s\in T$, the set
		\[
		\{A\in \mathcal{T}: s^\frown A\in T\}
		\]
		is a uniform cover of $(X,\U)$. 
		
		If $T$ is a uniform covering tree, then for each $t\in T$, we set
		\[C_T(t)=\{A\in \mathcal{T}: t^\frown A\in T\}.\]
	\end{definition}
	
	We denote by $\UT(X,\U)$ (or simply $\UT(X)$) the set of all uniform covering trees. For each nonzero $n\in\omega$ and $s\in \mathcal{T}^n$, define a subset $W_s$ of $F(X,\U)$ by
	\[
	W_s:=\bigcup_{\varepsilon\in\{-1,1\}^n}
	s(0)^{-\varepsilon(0)}\cdots s(n-1)^{-\varepsilon(n-1)}s(n-1)^{\varepsilon(n-1)}\cdots s(0)^{\varepsilon(0)}.
	\]
	For $T\in \UT(X)$, set $W_T=\bigcup_{s\in T} W_s$.
	
	Finally, for a sequence $P=(P_0,P_1,\dots)\in \UT(X)^\omega$, define
	\[
	O_P=\bigcup_{n\in\omega\setminus\{0\}}\ \bigcup_{\pi\in \sym(n)} 
	W_{P_{\pi(0)}}\cdots W_{P_{\pi(n-1)}}.
	\]
	
	\begin{theorem}\cite[Theorem 3.6]{NT}
		For every uniform space $(X,\U)$, the family $\Sigma:=\{O_P: P\in \UT(X,\U)^\omega\}$ forms a local base at $e$ of the topological group $F(X,\U)$.
	\end{theorem}
	
	We are now ready to give the final result.
	
	\begin{definition}
		Let $\tau$ be an infinite cardinal. A uniform space $(X,\U)$ is called \emph{$\tau$-precompact} if for every $U\in \U$ there exists a subset $A\subseteq X$ with $|A| < \tau$ such that
		\[
		X=\bigcup_{a\in A}B(a,U).
		\]
	\end{definition}
	
	By definition, precompactness is the same as $\omega$-precompactness, while $\omega$-narrowness is equivalent to $\omega_1$-precompactness.
	
	\begin{theorem}\label{main}
		Let $(X,\U)$ be a non-discrete uniform space. If $(X,\U)$ is $\add(\U)$-precompact, then
		\[
		\Ne_e(F(X, \U))=_T \U^\omega.
		\]
	\end{theorem}
	
	\begin{proof}
		All countable powers in this proof are ordered coordinatewise. Recall that every entourage contains an open entourage, so the family $\mathcal U_o$ of open entourages is cofinal in $\mathcal U$. Put $\tau=\add(\mathcal U)$. Since $(X,\mathcal U)$ is non-discrete, $\tau$ is infinite. For each $U\in\mathcal U_o$, choose $Y_U\subseteq X$ such that $|Y_U|<\tau$ and
		\[
		X=\bigcup_{y\in Y_U}B(y,U).
		\]
		Set $\gamma_U=\{B(y,2U):y\in Y_U\}$. Every $U$-ball is contained in some member of $\gamma_U$. Moreover,
		$
		B(y,2U)=\bigcup_{z\in B(y,U)}B(z,U),
		$
		so the members of $\gamma_U$ are open. Thus $\gamma_U$ is an open uniform cover and $|\gamma_U|<\tau$.
		
		Let $\operatorname{Cov}(\mathcal U)$ be the family of all uniform covers of $(X,\mathcal U)$. For $\mathcal C,\mathcal D\in\operatorname{Cov}(\mathcal U)$, write $\mathcal C\le_{\mathrm r}\mathcal D$ if $\mathcal D$ refines $\mathcal C$. Let
		$
		\Gamma=\{\gamma_U:U\in\mathcal U_o\},
		$
		with the order inherited from $\operatorname{Cov}(\mathcal U)$. Thus equal covers obtained from different entourages are identified.
		
		\begin{claim}\label{c0}
			$\Gamma=_T\mathcal U$.
		\end{claim}
		
		\begin{proof}[Proof of Claim~\ref{c0}]
			For $U\in\mathcal U$, let $\beta_U=\{B(x,U):x\in X\}$. The map $U\mapsto\beta_U$ is order-preserving and has cofinal image in $\operatorname{Cov}(\mathcal U)$.

			For $\mathcal C\in\operatorname{Cov}(\mathcal U)$, put
			$
			E_{\mathcal C}=\bigcup\{A\times A:A\in\mathcal C\}.
			$
			Since some ball cover refines $\mathcal C$, the set $E_{\mathcal C}$ contains an entourage and hence belongs to $\mathcal U$. The map $\mathcal C\mapsto E_{\mathcal C}$ is order-preserving. If $U\in\mathcal U$, choose $V\in\mathcal U$ with $2V\subseteq U$. Then $E_{\beta_V}\subseteq2V\subseteq U$. Hence this map also has cofinal image, and therefore $\operatorname{Cov}(\mathcal U)=_T\mathcal U$.

			It remains to note that $\Gamma$ is cofinal in $\operatorname{Cov}(\mathcal U)$. Let $\mathcal C$ be a uniform cover. Choose $E\in\mathcal U$ such that $\beta_E$ refines $\mathcal C$. Next choose $V\in\mathcal U$ with $2V\subseteq E$, and then $U\in\mathcal U_o$ with $U\subseteq V$. Since $\gamma_U$ refines $\beta_E$, it also refines $\mathcal C$. Thus $\Gamma$ is a cofinal subset of $\operatorname{Cov}(\mathcal U)$, and the claim follows.
		\end{proof}
		
		By Claim~\ref{c0}, $\add(\Gamma)=\tau$. Call a uniform covering tree $T$ \emph{neat} if
		\[
		C_T(s)=C_T(t)\in\Gamma
		\]
		whenever $s,t\in T$ have the same length, where $C_T(s)$ is as defined in Definition~\ref{def}. Let $\NT$ be the collection of all neat trees. For each $T\in\NT$, define $f_T\in\Gamma^\omega$ by $f_T(n)=C_T(s)$ for any $s\in T$ of length $n$. This gives a bijection between $\NT$ and $\Gamma^\omega$. We give $\NT$ the corresponding product order.
		
		\begin{claim}\label{c2}
			The family
			\[
			\Sigma^*=\{O_P:P\in\NT^\omega\}
			\]
			is a local base at $e$ in $F(X,\U)$.
		\end{claim}
		
		\begin{proof}[Proof of Claim~\ref{c2}]
			It is enough to show that for every $T\in\UT(X)$ there exists $S\in\NT$ such that $W_S\subseteq W_T$.
			
			Fix $T\in\UT(X)$. We construct a tree $S$, a map $h:S\to T$, and a sequence $(\delta_n)_{n\in\omega}\in\Gamma^\omega$ simultaneously. The map $h$ will preserve lengths and initial segments, and the construction will satisfy
			\[
			C_S(s)=\delta_{\lh(s)}
			\quad\text{and}\quad
			s(i)\subseteq h(s)(i)\quad(i<\lh(s))
			\tag{$*$}
			\]
			for every $s\in S$.
			
			Let $h(\emptyset)=\emptyset$. Suppose that $S$ and $h$ have been defined through level $n$, and that $\delta_0,\ldots,\delta_{n-1}$ have been chosen so that the first condition in $(*)$ holds below level $n$. Then the $n$-th level of $S$ is
			\[
			L_n=\{s\in S:\lh(s)=n\}=\prod_{i<n}\delta_i,
			\]
			where the empty product is $L_0=\{\emptyset\}$. Since every $\delta_i\in\Gamma$ has cardinality less than $\tau$, and $\tau$ is infinite, it follows that $|L_n|<\tau$.
			
			For every $s\in L_n$, choose $\delta_s\in\Gamma$ which refines $C_T(h(s))$. The family $\{\delta_s:s\in L_n\}$ has cardinality less than $\add(\Gamma)$ and is therefore bounded in $\Gamma$. Hence there is $\delta_n\in\Gamma$ which refines every $\delta_s$. Put $C_S(s)=\delta_n$ for all $s\in L_n$. For each $s\in L_n$ and $A\in\delta_n$, choose $\widehat{A}_s\in C_T(h(s))$ such that $A\subseteq\widehat{A}_s$, and set
			$
			h(s^\frown A)=h(s)^\frown\widehat{A}_s.
			$
			This completes the construction at level $n+1$.
			
			The first condition in $(*)$ shows that $S$ is neat. The second one gives $W_s\subseteq W_{h(s)}\subseteq W_T$ for every $s\in S$. Hence $W_S\subseteq W_T$, as required.
		\end{proof}
		
		Suppose that $T,S\in\NT$ and $T\le S$. For every $s\in S$, one can choose $t\in T$ of the same length such that $s(i)\subseteq t(i)$ for each coordinate. Hence $W_T\supseteq W_S$.
		
		Consequently, if $P_1\le P_2$ in $\NT^\omega$, then $O_{P_1}\supseteq O_{P_2}$. Thus the map $P\mapsto O_P$ is order-preserving and, by Claim~\ref{c2}, has cofinal image. Taking countable coordinatewise powers in Claim~\ref{c0} gives $\Gamma^\omega=_T\mathcal U^\omega$. Therefore,
		\[
		\Ne_e(F(X,\U))
		\le_T
		\NT^\omega
		\cong
		(\Gamma^\omega)^\omega
		\cong
		\Gamma^\omega
		=_T
		\U^\omega.
		\]
		
		For the reverse inequality, note that $A(X,\U)$ is a quotient group of $F(X,\U)$. Hence, by Lemma~\ref{3eq},
		\[
		\Ne_e(F(X,\U))
		\ge_T
		\Ne_e(A(X,\U))
		=_T
		\U^\omega.
		\]
		This completes the proof.
	\end{proof}

The following result shows that the uniform structure of an infinite precompact space has the largest possible cofinal type among directed sets of the same weight.
\begin{lemma}\label{FG}\cite{FG} Let $(X, \U)$ be an infinite precompact uniform space of weight $\tau$. Then $\U=_T [\tau]^{<\omega}$.\end{lemma}

\begin{corollary}\label{compact}
Let $X$ be an infinite compact space of weight $\tau$. Then we have
\[\Ne_e(F(X))=_T \Ne_e(A(X)) =_T ([\tau]^{<\omega})^\omega.\]
\end{corollary}

\begin{proof} Let $\U$ be the unique compatible uniformity. Then $w(\U)=\tau$, and
$F(X)\cong F(X,\U)$ and $A(X)\cong A(X,\U)$. Since $X$ is infinite and compact, it is non-discrete and $\kappa$-precompact for every infinite cardinal $\kappa$, so Theorem~\ref{main} applies. Together with Lemmas~\ref{3eq} and~\ref{FG}, it gives
$
\Ne_e(F(X))=_T\Ne_e(A(X))=_T\U^\omega=_T([\tau]^{<\omega})^\omega,
$
as required.
\end{proof}

Before stating the next corollary, we briefly recall Shelah's Strong Hypothesis ({\bf SSH}) \cite{Sh}. Arising from pcf theory, {\bf SSH} asserts that the pseudo-power $pp(\lambda) = \lambda^+$ for every singular cardinal $\lambda$.

The character consequences below are already contained in the earlier computations of Nickolas--Tkachenko and Chiș--Ferrer--Hernández--Tsaban. More precisely, writing $\operatorname{Fin}(\tau)=[\tau]^{<\omega}$, \cite[Proposition 4.15]{CFHT} proves
\[
\cof\big(\operatorname{Fin}(\tau)^\omega\big)
=
\mathfrak d\cdot\cof([\tau]^{\aleph_0}),
\]
and \cite[Theorem 8.7]{CFHT} evaluates $\cof([\tau]^{\aleph_0})$ under \textbf{SSH}. Thus the following lemma is only a convenient restatement of those known cardinal computations.

\begin{lemma}\label{cofp}\cite[Proposition 4.15 and Theorem 8.7]{CFHT} Assume \textbf{SSH}. Let $\tau>\omega$. Then
\begin{itemize}
    \item $\cof\big(([\tau]^{<\omega})^\omega\big) = \mathfrak{d}\cdot \tau$, if $\cof(\tau)>\omega$;
    \item $\cof\big(([\tau]^{<\omega})^\omega\big) = \mathfrak{d}\cdot \tau^+$, if $\cof(\tau)=\omega$.
\end{itemize}
\end{lemma}

The next corollary is included for completeness; its character formulas are not new. Indeed, \cite[Theorems 2.9 and 3.15]{NT} compute the relevant characters of $A(X)$ and $F(X)$ for compact spaces, while the required pcf evaluation follows from \cite[Proposition 4.15 and Theorem 8.7]{CFHT}. The new contribution of the present paper is instead the Tukey-level statement in Corollary~\ref{compact}, which identifies the full cofinal types of the neighborhood filters rather than merely their cofinalities.

\begin{corollary} Assume \textbf{SSH}. Let $X$ be an infinite compact space of weight $\tau$. Then
\begin{itemize}
    \item $\chi(F(X)) = \chi(A(X)) = \mathfrak{d}\cdot \tau$, if $\cof(\tau)>\omega$;
    \item $\chi(F(X)) = \chi(A(X))  = \mathfrak{d}\cdot \tau^+$, if $\cof(\tau)=\omega$.
\end{itemize}
\end{corollary}

\begin{proof}
By Corollary~\ref{compact}, both characters are equal to
$
\cof\big(([\tau]^{<\omega})^\omega\big).
$
For $\tau>\omega$, the conclusion follows from Lemma~\ref{cofp}. If $\tau=\omega$, then $[\omega]^{<\omega}=_T\omega$, and hence $([\omega]^{<\omega})^\omega=_T\omega^\omega$. Its cofinality is $\mathfrak d$, which also equals $\mathfrak d\cdot\omega_1$.
\end{proof}

	\begin{corollary}\label{SSH}
		Assume \textbf{SSH}. Suppose that $(X,\U)$ is an infinite precompact uniform space of weight $\tau$. If $\tau\geq \mathfrak{d}$ and $\cof(\tau)>\omega$, then
		\[
		\Ne_e(F(X,\U))
		=_T
		\Ne_e(A(X,\U))
		=_T
		[\tau]^{<\omega}.
		\]
	\end{corollary}
	
	\begin{proof}
		According to Lemma \ref{FG}, every precompact uniform space $(X, \U)$ satisfies
		\[
		\U=_T[\tau]^{<\omega}.
		\]
		Hence Theorem~\ref{main} and Lemma~\ref{3eq} yield
		\[
		\Ne_e(F(X,\U))
		=_T
		\Ne_e(A(X,\U))
		=_T
		([\tau]^{<\omega})^\omega.
		\]
		
		It remains to observe that, by Lemma \ref{cofp}, if $\tau\geq \mathfrak{d}$ and
		$
		\cof(\tau)>\omega
		$
		then
		\[
		\cof\big(([\tau]^{<\omega})^\omega\big)=\tau.
		\]
		Therefore, by (v) of Example \ref{ex},
	$
		([\tau]^{<\omega})^\omega\le_T [\tau]^{<\omega},
	$
		while the reverse Tukey reduction is trivial. Hence
		\[
		([\tau]^{<\omega})^\omega =_T [\tau]^{<\omega}. ~~~\qedhere
		\]
	\end{proof}

Recall that a uniform space $(X, \U)$ is called a \emph{uniform $P$-space} if the intersection of every countable family of entourages in $\U$ belongs to $\U$. For a non-discrete uniform space, this is equivalent to $\add(\U)>\omega$. A discrete uniform space is also a uniform $P$-space, but in that case $\add(\U)=1$.
	\begin{corollary} If $(X, \U)$ is an $\omega$-narrow uniform $P$-space, then \[
		\Ne_e(F(X,\U))
		=_T
		\Ne_e(A(X,\U))
		=_T
		\U.
		\]
	\end{corollary}

	\begin{proof}
		If $X$ is discrete, then $F(X,\U)$ and $A(X,\U)$ are discrete and $\U=_T1$, so the assertion is immediate. Suppose that $X$ is non-discrete. Then $\add(\U)>\omega$. Since $(X,\U)$ is $\omega$-narrow, it is $\add(\U)$-precompact, and Theorem~\ref{main} gives $\Ne_e(F(X,\U))=_T\U^\omega$. Lemma~\ref{3eq} gives the same equivalence for $A(X,\U)$, while Example~\ref{ex}(vi) yields $\U^\omega=_T\U$.
	\end{proof}
	
	\begin{remark}
		We note that in their second paper \cite{NT2}, Nickolas and Tkachenko proved that for every infinite compact space $X$ of weight $\tau$,
		\[
		\chi(F(X))=\chi(A(X))=\mathfrak d\cdot \cof([\tau]^{<\omega_1}).
		\]
		
		One might ask whether this cardinal equivalence can be lifted to the Tukey equivalence
		\[
		\Ne_e(F(X))=_T \omega^\omega\times [\tau]^{<\omega_1}.
		\]
		We shall see that this can never happen when $\tau$ is uncountable.
		
	By Corollary~\ref{compact} it suffices to show that
		\[
		([\tau]^{<\omega})^\omega \nleq_T \omega^\omega\times [\tau]^{<\omega_1}.
		\]
		
		\begin{claim}\label{countub}
			Every unbounded subset of $([\tau]^{<\omega})^\omega$ contains a countable unbounded subset.
		\end{claim}
		
		\begin{proof}[Proof of Claim~\ref{countub}]
			Let $A\subseteq ([\tau]^{<\omega})^\omega$ be unbounded. Then for some $n\in\omega$, the set
			\[
			A(n)=\{a(n):a\in A\}
			\]
			is unbounded in $[\tau]^{<\omega}$. Since every infinite subset of $[\tau]^{<\omega}$ contains a countable unbounded subset, one can choose a countable $B\subseteq A$ such that
			\[
			B(n)=\{b(n):b\in B\}
			\]
			remains unbounded. Hence $B$ is unbounded in $([\tau]^{<\omega})^\omega$.
		\end{proof}

		We also need the fact that
		$
		([\tau]^{<\omega})^\omega\nleq_T\omega^\omega.
		$
		For each $\alpha<\omega_1$, let $a_\alpha\in([\tau]^{<\omega})^\omega$ be given by $a_\alpha(0)=\{\alpha\}$ and $a_\alpha(n)=\emptyset$ for $n>0$. Every infinite subset of $\{a_\alpha:\alpha<\omega_1\}$ is unbounded. On the other hand, every uncountable subset of $\omega^\omega$ contains a countably infinite bounded subset. Indeed, one can recursively choose nested uncountable families whose members agree on longer and longer initial segments, and then select one distinct function from each family. Now let $\psi:([\tau]^{<\omega})^\omega\to\omega^\omega$. If $\psi(\{a_\alpha:\alpha<\omega_1\})$ is countable, some fiber of $\psi$ is uncountable; if it is uncountable, apply the preceding observation to its image. In either case, some infinite subset of $\{a_\alpha:\alpha<\omega_1\}$ has bounded image. Thus no such $\psi$ is an unbounded map, proving the assertion.
		
		Now let
		\[
		\varphi=(\varphi_1, \varphi_2):
		([\tau]^{<\omega})^\omega\to
		\omega^\omega\times [\tau]^{<\omega_1}
		\]
		be any map.
		Since
		$
		([\tau]^{<\omega})^\omega\nleq_T \omega^\omega,
		$
		the map $\varphi_1$ does not send unbounded sets to unbounded sets. Thus there exists an unbounded set $A$ such that $\varphi_1(A)$ is bounded in $\omega^\omega$. By Claim~\ref{countub}, passing to a subset if necessary, we may assume $A$ is countable.
		
		Since $\add([\tau]^{<\omega_1})=\omega_1$, the countable set $\varphi_2(A)$ is bounded in $[\tau]^{<\omega_1}$. Hence $\varphi(A)$ is bounded in
		$
		\omega^\omega\times [\tau]^{<\omega_1},
		$
		showing that $\varphi$ cannot be a Tukey map. Therefore
		\[
		([\tau]^{<\omega})^\omega \nleq_T \omega^\omega\times [\tau]^{<\omega_1},
		\]
		and so
		\[
		\Ne_e(F(X))\neq_T \omega^\omega\times [\tau]^{<\omega_1}.
		\]
		
		On the other hand, the reverse reduction
		$
		\omega^\omega\times [\tau]^{<\omega_1}\le_T ([\tau]^{<\omega})^\omega
		$
		always holds. Indeed, the map
		\[
		f\longmapsto\left(\bigl(|f(n)|\bigr)_{n\in\omega},
		\bigcup_{n\in\omega}f(n)\right)
		\]
		from $([\tau]^{<\omega})^\omega$ to $\omega^\omega\times[\tau]^{<\omega_1}$ is order-preserving and has cofinal image. To see cofinality, given $g\in\omega^\omega$ and $A\in[\tau]^{<\omega_1}$, choose finite sets $f(n)$ such that $|f(n)|\geq g(n)$ and $A\subseteq\bigcup_{n\in\omega}f(n)$.
	\end{remark}

	\section{Final Remarks and Open Problems}
	
	We close by returning to two issues that remain unresolved. The first concerns Question~\ref{MainQ}. Since $A(X,\U)$ is a quotient group of $F(X,\U)$, Lemma~\ref{3eq} always gives
	$
	\U^\omega\le_T\Ne_e(F(X,\U)).
	$
	Thus only the reverse Tukey reduction is in question. If $X$ is discrete, the desired equivalence is immediate. For non-discrete spaces, Theorem~\ref{main} gives the reverse reduction whenever $(X,\U)$ is $\add(\U)$-precompact. In particular, it settles Question~\ref{MainQ} for every precompact uniform space and for every non-discrete $\omega$-narrow uniform space with $\add(\U)>\omega$. The remaining case is therefore that of non-discrete, non-precompact $\omega$-narrow spaces with $\add(\U)=\omega$.
	
	The proof of \cite[Theorem 3.10]{NT} passes through the character estimate
	$
	\chi(F(X,\U))\leq\mathfrak d\cdot\cof(\U^\omega).
	$
	This suggests a Tukey upper bound which is weaker than the equality in Question~\ref{MainQ}.
	
	\begin{question}\label{Q:parametric}
		If $(X,\U)$ is $\omega$-narrow, must
		$
		\Ne_e(F(X,\U))\le_T\omega^\omega\times\U^\omega
		$
		hold?
	\end{question}
	
	The cofinality of the poset on the right is $\mathfrak d\cdot\cof(\U^\omega)$, so this question is a direct Tukey counterpart of the preceding estimate. Moreover, whenever $\omega^\omega\le_T\U^\omega$, a positive answer would imply
	$
	\Ne_e(F(X,\U))=_T\U^\omega
	$
	and hence give a positive answer to Question~\ref{MainQ} in that case.
	
	The second issue comes from the fineness index. Questions~\ref{Q1} and~\ref{Q:f(P)} ask whether $\f(P)$ is always realized as the character of a non-metrizable topological group with a $P$-base. Since $\f(P)$ is regular, the construction in Proposition~\ref{poset2group} shows that this is equivalent to
	$
	\f(P)\le_T P.
	$
	Proposition~\ref{finepower} gives $\f(P)\le_T P^\omega$ for every directed poset $P$. Consequently, the answer is positive whenever $P^\omega=_T P$, including the classical case $P=\omega^\omega$. In view of Lemma~\ref{cont} and the countable-cofinality cases, the unresolved case is $\add(P)=\omega<\cof(P)$; it remains open whether the passage to the countable power is necessary there.

\section*{Acknowledgements}
We would like to thank Dr. Xuewei Ling for introducing the notion of topological
groups with an $\omega^\omega$-base, which served as the starting point of this work.
We are also grateful to Prof. Wei He for many valuable suggestions and helpful
comments throughout this research.
We thank the anonymous referee for a careful reading of an earlier version and for
pointing out several errors and gaps. The referee's comments led to substantial
improvements in both the correctness and the presentation of the paper.
This work was supported by NSFC Grants 12301089 and 12271258.

\section*{Use of Generative AI}
During the preparation of the revised manuscript, the authors used OpenAI's GPT to
assist in identifying possible gaps, checking the consistency of arguments, and
improving the exposition and language. In particular, GPT suggested the level-by-level
formulation used in the revised proof of Theorem~\ref{main}. The authors subsequently
rewrote the argument in their own notation and independently verified every
mathematical step. The authors take full responsibility for all content of the paper.

\end{document}